%% file: main.tex
\documentclass[12pt]{article} 

\usepackage{amsmath} 
\usepackage{amsthm}
\usepackage{amssymb}
\usepackage{enumitem} 
\usepackage{mathtools}
\usepackage[hidelinks]{hyperref}
\usepackage[utf8]{inputenc}
\usepackage{csquotes}
\usepackage[pagewise]{lineno}
\usepackage[english]{babel}
\usepackage{enumitem}
\usepackage{titlesec}
\usepackage{graphicx}
\usepackage{subcaption}
\usepackage{mwe}
\usepackage{bm}
\usepackage[title]{appendix}
\usepackage{tikz-cd}
\usepackage{caption}

\usepackage{array}  
\usepackage{multirow}  
\usepackage{geometry}  
\usepackage{diagbox}  

\usepackage{geometry}
\makeatletter
\BeforeBeginEnvironment{proof}{%
  \def\@Opargbegintheorem#1#2#3#4{#4\trivlist
      \item[]{#3#2\@thmcounterend\ }}%
}
\AfterEndEnvironment{proof}{%
  \def\@Opargbegintheorem#1#2#3#4{#4\trivlist
      \item[\hskip\labelsep{#3#1}]{#3(#2)\@thmcounterend\ }}%
}
 \newtheorem{thm}{Theorem}[subsection]
 \newtheorem{cor}[thm]{Corollary}
 \newtheorem{lem}[thm]{Lemma}
 \newtheorem{defn}{Definition}[subsection]
 \newtheorem{exmp}{Example}[subsection]
 \newtheorem{rema}{Remark}[subsection]
 \newtheorem*{sthm}{Theorem}
\hypersetup{
    colorlinks=true, 
    linktoc=all,     
    linkcolor = blue,
}

\newcommand{\lcm}{{\text{lcm}}}

\providecommand{\keywords}[1]
{
	\small	
	\textbf{\text{Keywords:}} #1
}

\providecommand{\subclass}[1]
{
	\small	
	\textbf{\text{MSC(2020):}} #1
}

\begin{document}

	\title{\Large Further Generalization of Ramanujan Sums \\ with Regular $A$-Functions}
    \author{\normalsize Udvas Acharjee ,  N. Uday Kiran\\
    \begin{small}Department of Mathematics and Computer Science\end{small}\\
	\small Sri Sathya Sai Institute of Higher Learning, Puttaparthi, India.\thanks{Dedicated to Bhagawan Sri Sathya Sai Baba.}}
	\maketitle	
		
	\begin{abstract}
          
         In the study of Ramanujan sums, the so-called regular $A$-function is a set-valued multiplicative function that tracks certain subsets of the divisor sets of natural numbers. McCarthy provided a generalization of the Ramanujan sum using these regular $A$-function based arithmetic convolutions. This approach has recently attracted considerable interest from several researchers. In this paper, we extend McCarthy's generalization by introducing two regular $A$-functions corresponding to both parameters in the Ramanujan sum. Our generalization exhibits several properties of the Ramanujan sums. Using the regular $A$-function, We also generalize the greatest common divisor (GCD) function and the Von Sterneck formula. Our introduction of two regular $A$-functions into these expressions enables us to explore a novel perspective on the connection between these expressions and the order relation between the two regular $A$-functions. In particular, we establish the necessary and sufficient conditions for orthogonality and Dedekind-Hölder's identity (i.e., Ramanujan sum = Von Sterneck function) to hold. Our primary motivation for the proposed generalization is Ramanujan expansions of arithmetic functions based on arbitrary regular $A$-functions. These expansions of arbitrary $A$-functions discussed here are new in the literature.

        
		\subclass{11A25 $\cdot$ 11N64 $\cdot$ 11N37}
	\end{abstract}
\section*{Introduction}
\input{Introduction}
\section{Generalization of Ramanujan sum}\label{defn of generalization and properties}
\input{Generalization}
\section{Expansions of arithmetic functions of several variables}\label{exp_gen}
\input{Expansions}

\section*{Acknowledgment}
The second author is supported by the SERB-MATRICS project (MTR/2023/000705) from the Department of Science and Technology, India, for this work.

\input{References}
\end{document}

%% file: Introduction.tex
In the literature, various generalizations of the Ramanujan sums have been explored, each depending on the specific choice of divisors selected for their parameters in the Dirichlet convolution. Various researchers have explored and interpreted these sums in several interesting and innovative ways. Works by Cohen \cite{cohen_ext,Cohen_1955_Additive_properties,Cohen_connections_with_totient_functions,Cohen1960_UnitaryRamSum}, Sugunamma \cite{Sugunamma-generalization}, McCarthy \cite{McCarthy1968,McCarthy_regular_arithmetical_convolution}, Ramaiah \cite{Ramaiah1978}, and T\'{o}th \cite{Recent_Progress,Toth_Modified_Ramsums}, along with many researchers, have shown that a slight change in the definition of the sum not only offers great flexibility but also retains the powerful properties of these sums. It is known from the work of Narkiewicz \cite{Narkiewicz1963} that a general type of divisor set that preserves the properties of the Dirichlet convolution is the regular divisor set, and the associated convolution is termed the regular convolution. The regular divisor set, which is more recently termed as the Regular $A$-function \cite{semi-regularconvolution,Burnett_A_function}, provides a reasonably broad framework for several divisibility relations including the unitary divisors and the complete set of divisors.

The regular $A$-function $A:\mathbb{N}\rightarrow 2^\mathbb{N}$ is a set-valued multiplicative function with $\{1\}\subseteq A(n)\subseteq D(n)$, where $D(n)$ is the divisor set of $n$. We denote the collection of regular $A$-functions by $\mathbb{A}$. In the context of regular $A$-functions, McCarthy \cite{McCarthy1968,McCarthy_regular_arithmetical_convolution} defined the generalization of the Ramanujan sum $ c_{A}(m,n)$, considering divisors of $n$ from $A(n)$ for a specific choice of $A$. In \cite{Burnett_A_function}, the authors extended this generalization to the case of $ A $-functions being semi-regular.

In this paper, we extend the McCarthy's generalization $c_{A}(m,n)$ to the following generalization:
\begin{equation}\label{eq:A1A2Ramanujan sum}
    C_{(A_1,A_2)}(m,n)=\sum_{d\in A_1(m)\cap A_2(n)}d\mu_{A_2}\left(\frac{n}{d}\right),
\end{equation}    
where $A_{1},A_{2}\in \mathbb{A}$ and $\mu_{A}(\cdot)$ is the M\"{o}bius function associated with $A$. This generalization offers a broad framework for various existing Ramanujan sum generalizations in the literature. Indeed, if $A_{1}=D$, the divisor set, we recover $c_{A}(m,n)$. When $A_1=A_2=U$, the unitary divisor set, we obtain the modified unitary Ramanujan sum $\tilde{c}(m,n)$ given by T\'{o}th \cite{Toth_Modified_Ramsums}. We demonstrate that the above generalization retains several properties of the original Ramanujan sum, such as integrality, multiplicativity, orthogonality and so on. Along similar lines, we also generalize the gcd function as $\textnormal{gcd}_{(A_1, A_2)}(m, n)$, and thereby the Von Sterneck formula $ \Phi_{(A_1, A_2)}(m, n)$ (see Equation \ref{Von-Sterneck eq}).

Through the extension (\ref{eq:A1A2Ramanujan sum}), we can introduce a novel perspective on the relationship between \( A_1 \) and \( A_2 \). Specifically, we consider the natural order relation:
$$
A_2 \leq A_1 \quad \text{defined by} \quad A_2(n) \subseteq A_1(n) \text{ for all } n \in \mathbb{N}.
$$
Clearly, this relation forms a poset and the associated lattice $\mathbb{A}$ is distributive (see Theorem \ref{thm:lattice}). For $A_{2}\leq A_{1}$, McCarthy\cite{McCarthy1968} expressed $c_{A_{1}}(m,n)$ in terms of $c_{A_{2}}(m,n)$. In this work, we take a different line in terms of comparison between regular $A$-functions associated with the same sum. For instance, a comparison on the two regular $A$-functions leads to an elegant necessary and sufficient condition for the equality of generalized Ramanujan sum and Von Sterneck function. 

\begin{sthm}[\textnormal{\textbf{Generalized Dedekind-H\"{o}lder Identity}}]
    For $A_{1},A_{2}\in \mathbb{A}$,  $A_{2}\leq A_{1}$ is necessary and sufficient for the identity:  
    $
     C_{(A_1, A_2)}(m, n)=\Phi_{(A_1, A_2)}(m, n).
    $
\end{sthm}

Since H\"{o}lder \cite{Hölder_1936}, several generalizations of the Ramanujan sum have been associated with a generalization of the Von Sterneck function by Cohen \cite{Cohen_connections_with_totient_functions}, Suryanarayana \cite{DSuryaNarayana_VonSterneck}, and several others, including, very recently, Huang \cite{huang2008generalization}. Our result generalizes McCarthy \cite[Theorem 1]{McCarthy_regular_arithmetical_convolution}. Interestingly, although our Ramanujan sum generalization matches the Von Sterneck form, we show it cannot equal a similar generalization of trigonometric sum unless $A_1 = D$ as considered by McCarthy (see  \cite{McCarthy1968,semi-regularconvolution}).

Our motivation for (\ref{eq:A1A2Ramanujan sum}) is to obtain expansions of arithmetic functions based on regular $A$-functions.  For example, consider the arithmetic function based on $A$:
$$
g_{A}(n) = \frac{1}{n}\sum_{d \in A(n)}d.
$$
Using our sums we have the following neat expansion for the arbitrary regular $A$-function:
 $$
 g_{A}(n)=\zeta(2)\sum_{\substack{q=1}}^\infty\frac{\varphi_2(q)}{q^{4}}C_{(A,U)}(n,q),
 $$ 
 where $C_{A,U}(n,q)$ is given in (\ref{eq:A1A2Ramanujan sum}) and $U$ is the unitary divisor set. Note that the above expansion is for an arbitrary $A$. If $A=U$ then we obtain the sum in \cite[pg. no. 3]{Toth_Modified_Ramsums}. On the other hand, $ g_A $ doesn't have any simple expansions in terms of the existing generalizations in the literature. Indeed, consider two distinct primes $ p $ and $ q $, and a regular $ A $-function given by the unitary structure on $ p $: $ A(p^k) = \{ 1, p^k \} $ and the divisor set on $ q $: $ A(q^k) = \{ 1, q, \dots, q^k \} $ for $ k > 0 $. Then, using the Dirichlet convolution $( * )$ and the unitary convolution $( \times )$, one can show:
$$
(\mu * g_A)(p^r) = \frac{1}{p^r} - \frac{1}{p^{r-1}} \quad \text{and} \quad (\mu^* \times g_A)(q^r) = \frac{1}{q^r} + \frac{1}{q^{r-1}} + \dots + \frac{1}{q},
$$
for $ r > 1 $. Therefore, a simple arithmetic expansion cannot be obtained for $ g_A $, either through the original Ramanujan sum \cite{Ramanujan_expansion} (see \cite[Theorem 4.3]{T_th_2017_multivariable}) or the modified unitary Ramanujan sum \cite{Toth_Modified_Ramsums}, due to the presence of lower-order terms with different coefficients.

Results on Ramanujan sum expansions of standard multivariable arithmetic functions and their unitary analogues exist in the literature (see \cite{T_th_2017_multivariable,Toth_Modified_Ramsums,Cohen1960_UnitaryRamSum}). To the best of our knowledge, the expansions of arbitrary $A$-functions discussed here are new in the literature.

The rest of the paper is organized as follows: In Section \ref{defn of generalization and properties}, we define the generalization of Ramanujan sums and explore its properties, including a generalization of the gcd function using two regular $A$-functions. We also investigate properties influenced by the partial ordering on regular $A$-functions (Section \ref{ordering induced properties}). In Section \ref{exp_gen}, we apply these sums to the expansion of multivariable arithmetic functions, following the work of Delange \cite{Delange1976OnRE}, Ushiroya \cite{Ushiyora_expansion}, and Toth \cite{T_th_2017_multivariable}. 

%% file: Generalization.tex
The regular divisor sets were utilized by Narkiewicz \cite{Narkiewicz1963} to define the regular convolution. Subsequently, McCarthy \cite{McCarthy1968} incorporated this idea into the definition of Ramanujan sums, thereby obtaining a generalization that retains several properties identical to the classical Ramanujan sum. In \cite{Burnett_A_function} the regular $A$-function is defined as follows,
\begin{defn}\label{regular A function}
Let $A: \mathbb{N} \to 2^{\mathbb{N}}$ be such that $ \{1\} \subseteq A(n) \subseteq D(n) $, where $ D(n) = \{ d : d \mid n \} $, and the function $ A $ satisfies the following properties:

\begin{enumerate}
    \item (Multiplicativity) $ A(mn) = A(m) A(n) $ for $ \gcd(m,n) = 1 $.
    \item $ A(p^k) = \{ 1, p^t, p^{2t}, \dots, p^k \} $ for some $ t \mid k $ and $ p $ prime.
    \item If $ p^b \in A(p^k) $ for a prime $ p $ and $ b \geq 1 $, then $ A(p^b) = \{ 1, p^t, p^{2t}, \dots, p^b \} $.
\end{enumerate}

A number $n$ is said to be primitive in $A$ or $A$-primitive if $A(n)=\{1,n\}$. We also denote $\tau_A(\cdot)$ defined over prime powers to be the ``type" function (see \cite{Narkiewicz1963}) such that $p^{\tau_A(p^k)}$ is the primitive element in $A(p^k)$. 
\end{defn}

Clearly, the divisor set function $D$ is an example of a regular $A$-function. Another important example of a regular $A$-function is the unitary $A$-function $U$ defined as $U(n)=\{d: d\mid n, \gcd(d,n/d)=1\}$. We denote $d\in U(n)$ by $d\parallel n$. Cohen \cite{Cohen1960_UnitaryRamSum} introduced the Ramanujan sum based on $U$. It is known that regular $A$-functions generalizes the properties $D$ and $U$. In \cite{Narkiewicz1963} the regular convolution with the regular $A$-function $A$ is defined as:\begin{equation}\label{regular convolution}
f\times_Ag(n):=\sum_{d\in A(n)}f(d)g(n/d).    
\end{equation}
The following results are from \cite[Section 3]{Narkiewicz1963}.
\begin{thm}\label{regular convolution remark}
    The ring of arithmetic functions with ordinary addition and the regular convolution in Equation \ref{regular convolution}  associated with the regular $A$-function (Definition \ref{regular A function}) as multiplication, has the following properties:
    \begin{enumerate}
        \item It is associative, commutative and possesses a unit element defined by \[
        \iota(n)=\begin{cases}
            1,&\text{ if }n=1,\\
            0,&\text{otherwise}.
        \end{cases}
        \]
        \item The regular convolution preserves multiplicativity.
        \item The m\"{o}bius function $\mu_A$ given by the equation,\[
        1\times_A\mu_A=\iota,
        \] assumes for prime powers only the values $0,1$ and $-1$. More specifically,
        \begin{equation*}
    \mu_A(p^a)=\begin{cases}
        1,&\textnormal{ for }a=0\\
        -1, &\textnormal{ for }p^a \text{ is $A$-primitive}\\
        0, &\textnormal{ otherwise}
    \end{cases}
\end{equation*}
    \end{enumerate}
\end{thm} 
Clearly, for two arithmetic functions $f,g$ such that,
\[g(n)=(1\times_Af)(n),\]
the m\"{o}bius inversion for the regular convolution is given by \begin{equation} \label{mobius inversion}
    f(n)=(\mu_A\times_Ag)(n).
\end{equation}

\begin{table}[]
    \centering
    \begin{tabular}{|c|c|c|c|c|}
    \hline
         $A_1$ & $A_2$ &$\textnormal{gcd}_{(A_1,A_2)}(m,n)$  &$C_{(A_1,A_2)}(m,n)$ & $\Phi_{(A_1,A_2)}(m,n)$\\
         \hline
         \hline
         $D$& $D$& $(m,n)$&$c_n(m)$\cite{Ramanujan_expansion} & $\frac{\varphi(n)\mu(n/(m,n))}{\varphi(n/(m,n))}$\cite{Hölder_1936}
         \\[1.5 ex]
         \hline
         $D$ & $U$ & $(m,n)_*$& $c^*(m,n)$\cite{Cohen1960_UnitaryRamSum} & $\frac{\varphi^*(n)\mu^*(n/(m,n)_*)}{\varphi^*(n/(m,n)_*)}$ \cite{DSuryaNarayana_VonSterneck} \\[1.5 ex]
         \hline
         $D$ & $A$ & $(m,n)_A$ & $c_A(m,n)$\cite{McCarthy1968} & $\frac{\varphi_A(n)\mu_A(n/(m,n)_A)}{\varphi_A(n/(m,n)_A)}$\cite{McCarthy_regular_arithmetical_convolution} \\[1.5 ex]
         \hline
         $U$ & $U$ & $(m,n)_{**}$ & $\Tilde{c}(m,n)$\cite{Toth_Modified_Ramsums} & $\frac{\varphi^*(n)\mu^*(n/(m,n)_{**})}{\varphi^*(n/(m,n)_{**})}$ \\
         [1.5 ex]
         \hline
    \end{tabular}
        \captionsetup{justification=centering}
        \caption{Some Ramanujan sums in literature that fit within our framework.}
    \label{tab:my_label}
\end{table}

 \begin{defn}[Generalization of the Ramanujan sum]\label{definition of generalization of RSum} Let $A_1,A_2$ be two regular $A$-functions. The generalized Ramanujan sum is defined by 
\begin{equation*}
    C_{(A_1,A_2)}(m,n)=\sum_{d\in A_1(m)\cap A_2(n)}d\mu_{A_2}\left(\frac{n}{d}\right),
\end{equation*}    
\end{defn}This definition can be compared to the one given in \cite[Theorem 1]{semi-regularconvolution}, and it can be easily seen that when $A_1 = D,$ these sums are the ones given by McCarthy \cite{McCarthy_regular_arithmetical_convolution} \cite{McCarthy1968}. The table above, Table \ref{tab:my_label}, highlights the broad applicability of our generalization. In fact, several of our results address and fill gaps present in the existing specific cases. We now state some properties of our generalization.

\begin{thm}[Multiplicativity]\label{multiplicativity}
The following identities hold:\begin{enumerate}
    \item \label{point 1 mult} $C_{(A_1,A_2)}(m,n_1)C_{(A_1,A_2)}(m,n_2)=C_{(A_1,A_2)}(m,n_1n_2)$ for $\gcd(n_1,n_2)=1$.
    \item \label{point 2 mult}$\mu_{A_2}(n)C_{(A_1,A_2)}(m_1m_2,n)=C_{(A_1,A_2)}(m_1,n)C_{(A_1,A_2)}(m_2,n)$ for $\gcd(m_1,m_2)=1$ and $\mu_{A_2}(n)\neq 0$.
\end{enumerate} 
\end{thm}
\begin{proof}[Proof.]
To prove (\ref{point 1 mult}), first note that, using the multiplicativity property of regular $A$-functions,  $\{d_1d_2: d_1\in A_1(m)\cap A_2(n_1), d_2\in A_1(m)\cap A_2(n_2) \}$ is same as $\{d_1d_2: d_1d_2\in A_1(m)\cap A_2(n_1n_2)\}$. Therefore,  
    \begin{align*}
        C_{(A_1,A_2)}(m,n_1)C_{(A_1,A_2)}(m,n_2)=&\sum_{d_1\in A_1(m)\cap A_2(n_1)}d_1\mu_{A_2}\left(\frac{n_1}{d_1}\right)\sum_{d_2\in A_1(m)\cap A_2(n_2)}d_2\mu_{A_2}\left(\frac{n_2}{d_2}\right)\\
        =&\sum_{d_1d_2 \in A_1(m)\cap A_2(n_1n_2)}(d_1d_2)\mu_{A_2}\left(\frac{n_1n_2}{d_1d_2}\right)
    \end{align*}
    Now, we prove (\ref{point 2 mult}). Let $n=p^k$ for some prime $p$ such that $\mu_{A_2}(p^k)\neq 0$ i.e., $p^k$ is $A_2$-primitive.
    Then, \begin{align*}
        C_{(A_1,A_2)}(m_1,p^k)C_{(A_1,A_2)}(m_2,p^k)= &\sum_{d_1\in A_1(m_1)\cap A_2(p^k)}d_1\mu_{A_2}(p^k/d_1)\sum_{d_2\in A_1(m_2)\cap A_2(p^k)}d_2\mu_{A_2}(p^k/d_2).
    \end{align*}
    Now, as $(m_1,m_2)=1$ so either $d_1=1$ or $d_2=1$. This gives us,
\begin{eqnarray}
    \nonumber C_{(A_1,A_2)}(m_1,p^k)C_{(A_1,A_2)}(m_2,p^k)&=&\sum_{d\in A_1(m_1m_2)\cap A_2(p^k)}d\mu_{A_2}(p^k/d)\mu_{A_2}(p^k) \\
    \nonumber &=&\ \mu_{A_2}(p^k)C_{(A_1,A_2)}(m_1m_2,p^k).
\end{eqnarray}
The result follows from multiplcativity of $\mu_{A_2}(n)C_{(A_1,A_2)}(m,n)$ in $n$ for fixed $m$.  

\end{proof}

\begin{cor}[Evaluations at prime powers]\label{prime_eval}
    \begin{equation*}
        C_{(A_1,A_2)}(m,p^k)=\delta_{p^k}p^k-\delta_{p^{k-r}}p^{k-r},
    \end{equation*}
    where $r=\tau_{A_2}(p^k)$ and $\delta_x=\begin{cases}
        1, &x\in A_1(m)\\
        0, &x \notin A_1(m)
    \end{cases}$.
\end{cor}
\begin{proof}[Proof.]
 From the definition of $C_{A_{1},A_{2}}(m,n)$ we have 
 \begin{equation*}
     C_{(A_1,A_2)}(m,p^k)=\sum_{d\in A_1(m)\cap A_2(p^k)}d\mu_{A_2}\left(\frac{p^k}{d}\right).
 \end{equation*}
 Since $ \mu_{A_2} \left( \frac{p^k}{d} \right) $ is non-zero only for $ d = p^k $ or $ d = p^{k-r} $, the result follows.
\end{proof}

\begin{thm}[Sum over regular divisors]\label{res:sum_reg_div}
    \begin{equation*}
        \sum_{d\in A_2(n)}C_{(A_1,A_2)}(m,d)=\begin{cases}
            n, &n\in A_1(m)\\
            0, &n \notin A_1(m).
        \end{cases}
    \end{equation*}
\end{thm}
\begin{proof}[Proof.]
    Let the function $I_m^{(A_1)}:\mathbb{N}\rightarrow \mathbb{N}$ be defined as:\begin{equation*}
        I_m^{(A_1)}(x)=\begin{cases}
            x,&\text{ if }x\in A_1(m)\\
            0,&\text{ if }x\notin A_1(m)
        \end{cases}
    \end{equation*}
    Expressing the generalized Ramanujan sum in terms of $I_m^{(A_1)}(x)$ we have 
    \begin{align*}
        C_{(A_1,A_2)}(m,n)=&\sum_{d\in A_2(n)}I_m^{(A_1)}(d)\mu_{A_2}\left(\frac{n}{d}\right)\\
        =&(I_m^{(A_1)}\times_{A_2}\mu_{A_2})(n)
    \end{align*}
    The result follows by m\"{o}bius inversion given in Equation \ref{mobius inversion}.  
\end{proof}

\subsection{A Generalization of the GCD function}\label{gcd gen}

We extend GCD function with two regular $A$-functions.

\begin{defn}[A generalized GCD function]
    \begin{equation*}
        \text{gcd}_{(A_1,A_2)}(a,b)=\max\{d: d\in A_1(a)\cap A_2(b)\}.
    \end{equation*}
\end{defn}
When $A_1=D$ and $A_{2}$ is left arbitrary, this correspondence to the generalized gcd given in \cite{McCarthy1968,McCarthy_regular_arithmetical_convolution,semi-regularconvolution}. In the following, we discuss some key properties of gcd.
\begin{thm}[GCD function is multiplicative]
For $\gcd(n_1,n_2)=1$,
    \begin{equation*}
        \textnormal{gcd}_{(A_1,A_2)}(m,n_1)\textnormal{gcd}_{(A_1,A_2)}(m,n_2)=\textnormal{gcd}_{(A_1,A_2)}(m,n_1n_2).
    \end{equation*}
\end{thm}
\begin{proof}[Proof.]
    For two finite subsets $S, T$ of natural numbers, denote $ST=\{xy: x\in S, y\in T\}$. It is known that $\max(ST)=\max(S)\max(T)$. Let $B_{n_1}=\{d:d\in A_1(m)\cap A_2(n_1)\}$ and $B_{n_2}=\{d:d\in A_1(m)\cap A_2(n_2)\}$. So, $\max(B_{n_1}B_{n_2})=\max(B_{n_1})\max(B_{n_2})$. Since $(n_1,n_2)=1$, $B_{n_1}B_{n_2}=\{dd': d\in A_1(m)\cap A_2(n_1), d'\in A_1(m)\cap A_2(n_2) \}=\{dd': d\in A_1(m)\cap A_2(n_1n_2)\}$. 
\end{proof}
 
\begin{thm} Let $g=\textnormal{gcd}_{(A_1,A_2)}(m,n)$. Then, 
    \begin{equation*}
        d\in A_1(m)\cap A_2(n)\text{ if and only if } d\in A_1(g)\cap A_2(g).
    \end{equation*}
\end{thm}
\begin{proof}[Proof.]
First, note that $ d \in A_1(m) \cap A_2(n) $ if and only if $ p^k \in A_1(m) \cap A_2(n) $ for each $ p^k \parallel d $. If $ p^k \in A_1(m) \cap A_2(n) $, then it can be easily shown that $ p^k \in A_1(g) \cap A_2(g) $. Conversely, if $ p^k \in A_1(g) $, then $ p^k \in A_1(m) $, and similarly, $ p^k \in A_2(n) $.
\end{proof}
\begin{thm}\label{gcd prop 3}
    $
        \textnormal{gcd}_{(A_1,A_2)}(m,n)=d \text{ implies } \textnormal{gcd}_{(A_1,A_2)}(m/d,n/d)=1.
    $
\end{thm}
\begin{proof}[Proof.]
    It suffices to prove the prime power case using multiplicativity.  Consider 
    $A_1(p^{ks})=\{1,p^s,p^{2s},\dots, p^{ks}\}$ and $A_2(p^{rt})=\{1,p^r,p^{2r},\dots, p^{rt}\}$. Now there exists $\alpha,\beta$ such that $p^{\alpha s}=p^{\beta r}$ and no higher power of $p$ is in $A_1(p^{ks})\cap A_2(p^{rt})$. Now consider $S=A_1(p^{(k-\alpha)s})\cap A_2(p^{(t-\beta)r})$. If $S\neq \{1\}$ then there exist some $\gamma, \delta$ such that $p^{\gamma s}=p^{\delta r}\in S$. Then $p^{(\alpha+\gamma)s}\in A_1(p^{ks})$, $p^{(\beta+\delta)r}\in A_1(p^{rt})$ and $p^{(\alpha+\gamma)s}=p^{(\beta+\delta)r}$ and this contradicts the maximality of $p^{\alpha s}=p^{\beta r}$. 
    
\end{proof}
\begin{thm}
    Let $d\in A_1(m)\cap A_2(n)$. If $\textnormal{gcd}_{(A_1,A_2)}(m/d,n/d)=1$ then $\textnormal{gcd}_{(A_1,A_2)}(m,n)=d$.
\end{thm}
\begin{proof}[Proof.]
   It suffices to prove the prime power case. Note that $\text{gcd}_{(A_1,A_2)}(p^{sk},p^{tr})=\max(A_1(p^{ks})\cap A_2(p^{tr}))$ where $A_1(p^{ks}), A_2(p^{rt})$ is defined as in the proof of Theorem \ref{gcd prop 3}. Let $p^l\in A_1(p^{ks})\cap A_2(p^{tr})$ and $p^l\neq \text{gcd}_{(A_1,A_2)}(p^{ks},p^{tr})$ where $l=as=bt$ for some $a,b$.
   Now, if $p^{\alpha s}=p^{\beta r}=\text{gcd}_{(A_1,A_2)}(p^{ks},p^{tr})$, then \begin{align*}
       A_1(p^{(k-a)s})&=\{1, p^s,\dots,p^{(c-a)s},\dots, p^{(k-a)s}\} \quad\textnormal{and} \\
       A_2(p^{(t-b)r})&=\{1,p^r,\dots,p^{(d-b)r},\dots, p^{(t-b)r}\}.
   \end{align*}
   Clearly, $p^{(\alpha-a)s}=p^{(\beta-b)t}$ so $A_1(p^{(k-a)s})\cap A_2(p^{(t-b)r})\neq\{1\}$.
\end{proof}
Note that $ \text{gcd}_{(A_1, A_2)}(m/d, n/d) = 1 $ does not necessarily imply $ \text{gcd}_{(A_1, A_2)}(m, n) = d $ unless $ d \in A_1(m) \cap A_2(n) $. For example, for any $ n \neq m $ and $ d \mid \gcd(m, n) $, $ U(m/d) \cap U(n/d) = 1 $, but $ \text{gcd}_{(U, U)}(m, n) \neq d $ unless $ d = 1 $.
\begin{thm}
    Let $m,n\in A_1(t_1)\cap A_2(t_2)$ for some $m,n,t_1,t_2$ then $\textnormal{gcd}_{(A_1,A_2)}(m,n)=\gcd(m,n)$ and $\gcd(m,n)\in A_1(m)\cap A_2(n)$.
\end{thm}
\begin{proof}[Proof.]
    Since, $m,n\in A_1(t_1)\cap A_2(t_2)$ so if $p^t\parallel m$ and $p^s\parallel n$ then $p^t,p^s\in A_1(t_1)\cap A_2(t_2)$. So, either $p^t\in A_1(p^s)\cap A_2(p^s)$ or $p^s\in A_1(p^t)\cap A_2(p^t)$. This clearly entails that $\text{gcd}_{(A_1,A_2)}(m,n)=\gcd(m,n)$ and $\gcd(m,n)\in A_1(m)\cap A_2(n)$. 
\end{proof}
\subsection{Properties of the generalization of Ramanujan sum}\label{ordering induced properties}

We show that $\mathbb{A}$ is a lattice endowed with $\leq $ relation. 
   
\begin{lem}\label{partial order relation}
    Let $A_1$ and $A_2$ be two regular $A$-functions. Then, the following are equivalent:\begin{enumerate}
        \item \label{poset 1} $n\in A_1(m)$ implies $A_2(n)\subseteq A_1(m)$,
        \item \label{poset 2}$A_2(n)\subseteq A_1(n)$ for all $n\in \mathbb{N}$.
    \end{enumerate}
\end{lem}
\begin{proof}[Proof.]
    Clearly (\ref{poset 1}) implies (\ref{poset 2}) because $n\in A_1(n)$ due to reflexivity of regular $A$-functions. For the converse, note that $n \in A_1(m)$ implies $ A_1(n) \subseteq A_1(m)$ and from $(2)$ $A_2(n)\subseteq A_1(n)$. 
\end{proof}

For $A_{1},A_{2}\in \mathbb{A}$ we define the relation $A_{2}\leq A_{1}$ by $A_2(n)\subseteq A_1(n)$ for all $n\in \mathbb{N}$. 

\begin{thm}\label{thm:lattice}
    $\mathbb{A}$ is a distributive and complete lattice, with $D$ and $U$ being the maximal and minimal elements, respectively. 
\end{thm}
\begin{proof}[Proof.]
It is evident that the collection of regular $A$-functions $\mathbb{A}$ forms a poset under the relation $\leq$. It can also be easily shown that $\mathbb{A}$ is indeed a lattice with the meet and join defined for any two regular $A$-functions $A_1$ and $A_2$ as follows:
\begin{enumerate}
    \item $A_3=A_1\vee A_2$ such that $\tau_{A_3}(p^k)=\gcd(\tau_{A_1}(p^k),\tau_{A_2}(p^k))$,
    \item $A_4=A_1 \wedge A_2$ such that $\tau_{A_4}(p^k)=\lcm(\tau_{A_1}(p^k),\tau_{A_2}(p^k))$,
\end{enumerate}

It can be be easily shown that the lattice of regular $A$-functions is indeed distributive.
This is due to the following well known identities that says gcd and lcm are distributive,\[
\gcd(a,\lcm (b,c))=\lcm(\gcd(a,b),\gcd(a,c)),\]
\[\lcm(a,\gcd(b,c))=\gcd(\lcm(a,b),\lcm(a,c)).
\]

    We now show that $\mathbb{A}$ is complete i.e., every subset of regular $A$-functions has both a meet and join. Let $I$ be a set of indices which may be infinite. Note that for a fixed prime power $p^k$, $\tau_{A_i}(p^k), i\in I$ has only finitely many distinct values as $k$ has finitely many divisors. Call them $\{t_1,\dots, t_l\}$. Then define a new $A$- function with $\tau_A(p^k)=\gcd(t_1,\dots, t_k)$. It is clear that $A=\vee_{i\in I}A_i$. The argument for the existence of meet is similar.

    From the definition, for $k>1$ and $p$ prime we have 
    $$
    U(p^{k})=\{1,p^{k}\}\quad \textnormal{and}\quad D(p^{k})=\{1,p,\cdots,p^{k}\},
    $$
    which are two extremities. Therefore, $U(p^{k})\leq A(p^{k})\leq D(p^{k})$ for all $A\in \mathbb{A}$.     
     
\end{proof}

\begin{figure}
    \centering
    \includegraphics[width=0.4\linewidth]{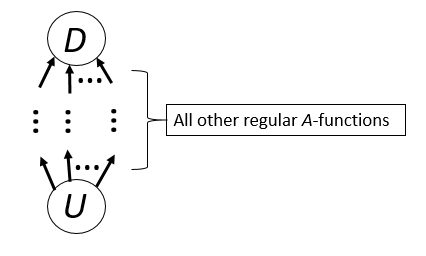}
    \caption{A depiction of the  poset of regular $A$-functions. Here $D$ denotes a complete set of divisors and $U$ is the set of unitary divisors.}
    \label{fig:struture A func}
\end{figure}

In what follows, we will relate the partial ordering on the regular $A$-functions to the existence of a Von Sterneck-like form for our generalization of Ramanujan sum. Von Sterneck(1902)  used the function,\[
\Phi(m,n)=\frac{\varphi(n)\mu(n/(m,n))}{\varphi(n/(m,n))},
\]
in certain combinatorial problems (see \cite{Von_Sterneck_Nicol}). The Dedekind-H\"{o}lders theorem \cite{Hölder_1936} states that the Von Sterneck's function is identical to the original Ramanujan sum. Several researchers generalized the Von Sterneck function and associated them with  generalizations of Ramanujan sum (see for instance \cite{Cohen_connections_with_totient_functions,DSuryaNarayana_VonSterneck,huang2008generalization,Haukkanen_generalization}). We define an analogue of the Von Sterneck's function as follows:\begin{equation}\label{Von-Sterneck eq}
    \Phi_{(A_1,A_2)}(m,n)=\frac{\varphi_{A_2}(n)\mu_{A_2}\left(\frac{n}{g}\right)}{\varphi_{A_2}\left(\frac{n}{g}\right)},
\end{equation}
where $g=\text{gcd}_{(A_1,A_2)}(m,n)$
and $\varphi_A(\cdot)$ is the $A$-function analogue of Euler totient function defined by:\begin{equation}\label{reg euler totient}
  \varphi_A(n)=|\{x: 1\leq x\leq n, D(x)\cap A(n)=\{1\}\}|.  
\end{equation}
It is can easily see that $\Phi_{A_{1},A_{2}}(m,n)$ is multiplicative. 

\begin{thm}[\textbf{Generalization of H\"{o}lder's relation}]\label{VonSterneck} Let $A_1,A_2$ be two regular $A$-functions. Then $A_2\leq A_1$ is necessary and sufficient for the identity $C_{(A_1,A_2)}(m,n)=\Phi_{(A_1,A_2)}(m,n)$ to hold. 
\end{thm}
\begin{proof}[Proof.]
    Again by multiplicativity, it suffices to prove this result for the prime power case only i.e., when $n=p^k$. Let $r=\tau_{A_2}(p^k)$, Assuming  $A_2\leq A_1$, from Corollary \ref{prime_eval} and Theorem \ref{partial order relation} it follows \[C_{(A_1,A_2)}(m,p^k)=\begin{cases}
            p^k-p^{k-r}, &p^k\in A_1(m)\\
            -p^{k-r}, &p^k\notin A_1(m),\ p^{k-r}\in A_1(m)\\
            0, &\text{ otherwise}
        \end{cases}\]
        Also notice that if $p^k\in A_1(m)$, $g=\text{gcd}_{(A_1,A_2)}(m,p^k)=p^k$
        and if $p^k\notin A_1(m), p^{k-r}\in A_1(m)$, $g=p^{k-r}$. It follows that \begin{equation*}
            \mu_{A_2}\left(\frac{p^k}{g}\right)=\begin{cases}
                1, &\text{ if }p^k\in A_1(m)\\
                -1, &\text{ if }p^{k}\notin A_1(m),\ p^{k-r}\in A_1(m)\\
                0, &\text{ otherwise}
            \end{cases}
        \end{equation*}
        So, in case $p^k\in A_1(m)$, $\varphi_{A_2}(p^k)=p^k-p^{k-r}$. In case $p^k\notin A_1(m),p^{k-r}\in A_1(m)$,\begin{align*}
            \frac{\varphi_{A_2}(p^k)}{\varphi_{A_2}(p^r)}=&\frac{p^k-p^{k-r}}{p^r-1}=p^{k-r}
        \end{align*}\\
        Conversely, suppose $A_{2}\nleq A_{1}$. Assume that for some $m$ and prime-power $p^k$, $p^k\in A_1(m)$, $p^{k-r}\notin A_1(m)$. The existence of such an $m$ and $p^k$ is guaranteed as $A_2\nleq A_1$. So  from Corollary \ref{prime_eval}, $C_{(A_1,A_2)}(m,p^k)=p^k$,$\varphi_{A_2}(p^k)=p^k-p^{k-r}$ and $\gcd_{(A_1,A_2)}(m,p^k)=p^k$. The inequality can be verified directly.  
\end{proof}
Since $A\leq D$ for every $A \in \mathbb{A}$ (see Figure \ref{fig:struture A func}), Theorem \ref{VonSterneck} for $A_1=D$  and arbitrary regular $A_2$ gives,\[
C_{(D,A_2)}(m,n)=\frac{\varphi_{A_2}(n)\mu_{A_2}(n/\gcd_{(D,A_2)}(m,n))}{\varphi_{A_2}(n/\gcd_{(D,A_2)}(m,n))},
\]
which is the identity given by McCarthy \cite[Theorem 1]{McCarthy_regular_arithmetical_convolution}. 
Also, for every $A\in \mathbb{A}$, $U\leq A$ so we can state the following corollary:
\begin{cor}
    For $A\in \mathbb{A}$ we have $
    C_{(A,U)}(m,n)=\Phi_{(A,U)}(m,n). 
    $
\end{cor}
Surprisingly, the partial ordering of regular $A$-functions can also be related to another property of the generalized Ramanujan sums stated below: 
\begin{thm}[Orthogonality]\label{orthogonality result}
   Let $A_1,A_2$ be regular $A$-functions then the condition $A_2\leq A_1$ is necessary and sufficient for:
     \begin{equation}\label{orthogonality}
        \sum_{d\in A_2(n)}C_{(A_1,A_2)}(n/d,\delta)C_{(A_1,A_2)}(n/\gamma,d)=\begin{cases}
            n, &\delta=\gamma\\
            0, &\delta\neq \gamma
            
        \end{cases} 
    \end{equation}
    for $\delta,\gamma\in A_2(n)$.
\end{thm}
\begin{proof}[Proof.]
    From Theorem \ref{multiplicativity} and Corollary \ref{prime_eval} it follows that \[C_{(A_1,A_2)}(m,n)=\prod_{p\in \mathbb{P}}C_{(A_1,A_2)}(p^{\nu_p(m)},p^{\nu_p(n)}),\] 
    where $\mathbb{P}$ is the collection of primes and $\nu_p(\cdot)$ is the p-adic valuation function. This means the LHS in Equation \ref{orthogonality} can be written as,\[
    \prod_{p\in \mathbb{P}}\sum_{d\in A_2(n)}C_{(A_1,A_2)}(p^{\nu_p(n)}/p^{\nu_p(d)},p^{\nu_p(\delta)})C_{(A_1,A_2)}(p^{\nu_p(n)}/p^{\nu_p(\gamma)},p^{\nu_p(d)}).
    \] So, it is enough if we prove this result for the prime case only i.e., for $n=p^{k}.$ Let $A_2(p^k)=\{1,p^r,p^{2r},\dots, p^k\}$ and $p^s,p^t\in A_2(p^k)$. We evaluate the expression by for the case $s=t$ only as the other cases namely $s<t$ and $s>t$ can be evaluated in similar lines.
    
    For $s=t$, consider the left hand side of Equation \ref{orthogonality} at $n=p^{k}$. Note that when $t=0$ then the result follows from Theorem \ref{res:sum_reg_div}.
    When $l>k-t+r$ then from Corollary \ref{prime_eval} it follows that $C_{(A_1,A_2)}(p^{k-t},p^l)=0$. When $l=k-t+r$, the left hand side of Equation \ref{orthogonality} becomes \[
    C_{(A_1,A_2)}(p^{t-r},p^t)C_{(A_1,A_2)}(p^{k-t},p^{k-t+r})=p^{t-r}p^{k-t}.
    \]
    Also, for every $l\leq k-t$,\begin{align*}
    C_{(A_1,A_2)}(p^{k-l},p^t)&=p^t-p^{t-r}\quad \text{ and }\\
    C_{(A_1,A_2)}(p^{k-t},p^l)&=\begin{cases}
        1, &\text{ for }l=0,\\
        p^l-p^{l-r}, &\text{ for }l\neq0.
    \end{cases}
    \end{align*}
    Combining these terms we obtain
    \begin{align*}
        \sum_{l\leq k-t}&C_{(A_1,A_2)}(p^{k-l},p^t)C_{(A_1,A_2)}(p^{k-t},p^l)+C_{(A_1,A_2)}(p^{t-r},p^t)C_{(A_1,A_2)}(p^{k-t},p^{k-t+r})\\
            &= (p^t-p^{t-r})(1+p^r-1+p^{2r}-p^r+\dots+ p^{k-t}-p^{k-t-r})+p^{t-r}p^{k-t}\\
            &= (p^t-p^{t-r})p^{k-t}+p^{t-r}p^{k-t} = p^k.
    \end{align*}
    We will prove the converse by showing that if $A_2\nleq A_1$ then Equation \ref{orthogonality} doesn't hold.
    We assume $A_2\nleq A_1$, so there is a prime $p$ and $k>1$ such that $A_2(p^k)\nsubseteq A_1(p^k)$. For the prime $p$ let $l>1$ be the least integer with this property i.e.,\[
    A_2(p^l)\nsubseteq A_1(p^l).
    \] 
    Let
    $$
    A_2(p^l)=\{1,p^s,p^{2s},\dots, p^l\}\quad \textnormal{ and }\quad A_1(p^l)=\{1,p^r,p^{2r},\dots, p^l\}.
    $$
    Since, $A_2(p^l)\nsubseteq A_1(p^l)$, so $r\nmid s$. By the minimality of $l$ we must have $l=\lcm(r,s)$ and $
    A_2(p^{l'})\subseteq A_1(p^{l'})$, when $l'<l$.
    
    We now evaluate Equation \ref{orthogonality} with $n=p^l, \delta=p^s\text{ and } \gamma=1$.
    \begin{equation}\label{eq:prime_orthogonality_case_2}
    \sum_{d\in A_2(p^l)}C_{(A_1,A_2)}(p^l/d,p^s)C_{(A_1,A_2)}(p^l,d)
    \end{equation}
    
    Since $l=\lcm(r,s)$ we have $p^{ks}\notin A_1(p^l)$ for $1\leq k<l/s$, which means $C_{(A_1,A_2)}(p^l,d)$ is non-zero only for $d\in \{1,p^s,p^l\}$. Hence, summing the non-zero terms in Equation \ref{eq:prime_orthogonality_case_2}.
    \begin{align*} 
        &C_{(A_1,A_2)}(p^l,p^s)C_{(A_1,A_2)}(p^l,1)+C_{(A_1,A_2)}(p^{l-s},p^s)C_{(A_1,A_2)}(p^l,p^s)+C_{(A_1,A_2)}(1,p^s)C_{(A_1,A_2)}(p^l,p^l)\\
        &=(-1)1+(p^s-1)(-1)+(-1)p^l\\
        &=-p^l-p^s\neq 0
    \end{align*}  
\end{proof} 
A special case of Theorem \ref{orthogonality result} with $A_1 = A_2 = U$ gives a modified unitary sum orthogonality relation by T\'{o}th \cite{Toth_Modified_Ramsums}, which, to the best of our knowledge, is not known in the literature.

\begin{thm}\label{thm: bound}
    If $A_2\leq A_1$ then \begin{equation*}
        \sum_{d\in A_2(n)}|C_{(A_1,A_2)}(m,d)|=2^{\omega(\frac{n}{g})}g,
    \end{equation*}
    where $g=\textnormal{gcd}_{(A_1,A_2)}(m,n)$ and $\omega(n)$ denotes the function that returns the number of distinct primes in the prime factorization of $n$ with $\omega(1)=0$.
\end{thm}
\begin{proof}[Proof.]
    The result is clearly true for $n=1$. We will verify the result for $n=p^k$ and the rest will follow from multiplicativity.
    
    Let $\text{gcd}_{(A_1,A_2)}(m,p^k)=p^{br}$ where $r=\tau_{A_2}(p^k)$. Therefore when we sum the non-zero terms using Corollary \ref{prime_eval} we get: \begin{align*}
        &|C_{(A_1,A_2)}(m,1)|+|C_{(A_1,A_2)}(m,p^r)|+\dots+|C_{(A_1,A_2)}(m,p^{br})|+|C_{(A_1,A_2)}(m,p^{r(b+1)})|\\
        & =1+(p^r-1)+(p^{2r}-p^r)+\dots+(p^{br}-p^{br-r})+p^{br}=2p^{br}.
    \end{align*}

This implies $\sum_{d\in A_2(n)}|C_{(A_1,A_2)}(m,d)|= 2^{\omega(n)}m$.
\end{proof}

\subsection{Trigonometric sum}
It is known that \begin{equation}\label{exp_form}
    C_{D,A}(m,n)=\sum_{d\in D(m)\cap A(n)}d\mu_A\left(\frac{n}{d}\right)=\sum_{\substack{k=1\\\gcd_{(D,A)}(k,n)=1}}^n e^{\frac{2\pi i k m}{n}},
\end{equation}
for an arbitrary regular $A$-function $A$ (see \cite[Theorem 1]{semi-regularconvolution}). We now show that this is the only case which is of the form of a trigonometric sum. For this purpose, we slightly generalize the result in \cite[Lemma 1]{semi-regularconvolution}, and the proof follows similar lines.

\begin{lem} \label{sum_mobius_function}
    If $A_2\leq A_1$ then,\[
    \sum_{d\in A_1(k)\cap A_2(n)}\mu_{A_2}(d)=\begin{cases}
        1, &\text{ if }A_1(k)\cap A_2(n)=\{1\}\\
        0, &\text{ otherwise}
    \end{cases}
    \]
\end{lem}

It should also be noted that the above results doesn't hold if $A_2\nleq A_1$. For example, If $A_1(p^2)=\{1,p^2\}$ and $A_2(p^2)=\{1,p,p^2\}$ then $A_1(p^2)\cap A_2(p^2)=\{1,p^2\}$. But, $\mu_{A_2}(1)+\mu_{A_2}(p^2)=1\neq 0$.

For $A_2\leq A_1$, by Lemma \ref{sum_mobius_function} we have 
\[
\sum_{\substack{k=1\\A_1(k)\cap A_2(n)=\{1\}}}^n e^{2\pi ikm/n}=\sum_{k=1}^n\sum_{d\in A_1(k)\cap A_2(n)}\mu_{A_2}(d) e^{2\pi ikm/n}=\sum_{d\in A_2(n)}\mu_{A_2}(d)\sum_{\substack{k=1\\ d\in A_1(k)}}^n e^{2\pi ikm/n}
\]
If $A_1=D$ then the inner sum $\sum_{\substack{k=1\\ d\in A_1(k)}}^n e^{2\pi ikm/n}= n/d$. However, this is not the case for arbitrary $A_1$ such that $A_2\leq A_1$. Infact, it turns out that $D$ is the only regular $A$-function for which this equality holds. To show this, we need the following result by So \cite[Theorem 3.1 and 7.1]{WasinSo}: \begin{equation}\label{so's result}
 \text{For $S$ the sum},\  \sum_{k\in S}^ne^{2\pi imk/n}\in \mathbb{Z} \text{ for all } m \text{ if and only if } S=\cup_{d\in B(n)}G_n(d), \end{equation}
where $B(n)$ is a subset of divisors of $n$ and $G_n(d)=\{k: \gcd(k,n)=d, 1\leq k\leq n\}$.
\begin{thm}\label{trig_form}
    For $A_{1},A_{2}\in \mathbb{A}$, consider the natural trigonometric sum    \[S_{(A_1,A_2)}(m,n)=\sum_{\substack{k=1\\ \textnormal{gcd}_{(A_{1},A_{2})}(k,n)=1}}^ne^{2\pi i\frac{km}{n}}.\]  If $S_{(A_1,A_2)}(m,n)=C_{(A_1,A_2)}(m,n)$ then $A_1=D$. Moreover , there exists $m,n \in \mathbb{N}$ such that $S_{(A_1,A_2)}(m,n)\notin \mathbb{Z}$ if $A_1\neq D$.
\end{thm}
\begin{proof}[Proof.]
    We prove this result by  showing that if $A_1\neq D$ then there exists $m,n$ such that $S_{(A_1,A_2)}(m,n)$ is not an integer. If $A_1\neq D$ then for some prime $p$, then there exists $r>1$ such that $p^r$ is $A_1$-primitive. Let $q\neq p$ be another prime, choose $t>0$ such that $pq^t>p^r$. Denote $S=\{k: A_1(k)\cap A_2(pq^t)=\{1\}\}$. Note that $A_1(p^r)=\{1,p^r\}$ and $p^r\notin A_2(pq^t)$ so $p^r\in S$. Also, $p\in A_1(p)\cap A_2(pq^t)$ so $p\notin S$. Using Equation \ref{so's result},  $S$ is a union of $G_n(d)$'s then since $p^r\in S$, $G_n(p)$ must be one of the sets involved. But that would mean $p\in S$ which is a contradiction.   
\end{proof}

%% file: Expansions.tex
Expansions of arithmetic functions were introduced by Ramanujan in his seminal work \cite{Ramanujan_expansion}. Since then, numerous researchers have explored their generalizations and the conditions under which such expansions exist for arithmetic functions. Significant contributions to this field include the pioneering works of Wintner \cite{Wintner}, Delange \cite{Delange1976OnRE}, Lucht \cite{lucht_expansion}, and Spilker \cite{spilker1980ramanujan}. We also direct the reader to \cite{murty2013ramanujan} for a comprehensive survey.

An extension of arithmetic functions with two variables was recently given by Ushiroya \cite{Ushiyora_expansion}, who also provided a necessary condition for the existence of these expansions, similar to the one-variable case in Delange's work \cite{Delange1976OnRE}. Tóth \cite{T_th_2017_multivariable} further expanded this work to encompass arithmetic functions with an arbitrary number of variables, presenting several significant results related to these expansions. Noting that these expansions cannot be expressed using simple coefficients, Tóth defined a modified unitary Ramanujan sum and thereby provided expansions for certain arithmetic functions \cite{Toth_Modified_Ramsums}.

As stated in the introduction, existing theorems fail to provide neat expansions for the regular $A$-functions analogous to arithmetic functions. In this section, we develop the theory to provide such expansions using our generalized of Ramanujan sum. Since regular $A$-functions generalize the properties of the complete divisor set $D$ and unitary divisor $U$, the expansions in \cite{T_th_2017_multivariable} and \cite{Toth_Modified_Ramsums} will be specific cases of our theorems.

Let $\mathcal{A}_k$ 
 denote the set of functions $f:\mathbb{N}^k\rightarrow \mathbb{C}$ that forms a unital commutative ring under pointwise addition and regular convolution defined by:\begin{equation}\label{reg_convo}
    f\times_A g(n_1,\dots,n_k)=\sum_{d_1\in A(n_1),\dots,d_k\in A(n_k)}f(d_1,\dots,d_k)g(n_1/d_1,\dots,n_k/d_k),
\end{equation}
\\
with the unity being defined as:\begin{equation*}
    \delta_k(n_1,\dots,n_k)=\begin{cases}
        1, & n_1=n_2=\dots=n_k=1,\\
        0, & \text{otherwise.}
    \end{cases}
\end{equation*}
The inverse of the constant function $1$ under (\ref{reg_convo}) is given by:
\begin{equation*}
    \mu_{(A,k)}(n_1,\dots,n_k)=\mu_A(n_1)\dots \mu_A(n_k),
\end{equation*}
where, $\mu_A$ is the m\"{o}bius function associated to the regular $A$ convolution. In this context, we state a theorem on Ramanujan expansion that generalizes the a fundamental known result for Ramanujan sums \cite[Theorem 2]{T_th_2017_multivariable}.

\begin{thm}\label{expansion1}
    Let $f:\mathbb{N}^k\rightarrow \mathbb{C}$ and $A_1,A_2\in \mathbb{A}$  such that $A_2\leq A_1$. If,\begin{equation*}
        \sum_{q_1,\dots,q_k=1}^\infty 2^{\omega(q_1)+\dots+\omega(q_k)}\frac{|\mu_{(A_1,k)}\times_{A_1}f(q_1,\dots,q_k)|}{q_1\dots q_k}<\infty,
    \end{equation*}
    Then, 
    \begin{equation*}
        f(n_1,\dots,n_k)=\sum_{q_1\dots q_k=1}^\infty a(q_1,\dots, q_k)C_{(A_1,A_2)}(n_1,q_1)\dots C_{(A_1,A_2)}(n_k,q_k),
    \end{equation*}
    where
    \begin{equation*}
        a(q_1,\dots,q_k)=\sum_{\substack{m_1,\dots, m_k=1\\ q_i\in A_2(m_iq_i)}}^\infty\frac{\mu_{(A_1,k)}\times_{A_1}f(m_1q_1,\dots, m_kq_k)}{m_1q_1\dots m_kq_k}.
    \end{equation*}
\end{thm}
\begin{proof}[Proof.]
Since, $\mu_{(A_1,k)}\times_{A_1}1=\delta_k$ we have\[
f(n_1,\dots, n_k)= \mu_{(A_1,k)}\times_{A_1} f\times_{A_1}1(n_1,\dots, n_k).
\]
Treating $\mu_{(A_1,k)}\times_{A_1}f$ as one function and taking the regular convolution with $1$ we get,\begin{equation*}
    \sum_{d_1\in A_1(n_1),\dots, d_k\in A_1(n_k)}\mu_{(A_1,k)}\times_{A_1}f(d_1,\dots,d_k).
\end{equation*}
Now, using Theorem \ref{res:sum_reg_div}, we get the form, \begin{equation*}
    \sum_{d_1\dots d_k=1}^\infty\frac{\mu_{(A_1,k)}\times_{A_1}f(d_1,\dots,d_k)}{d_1\dots d_k}\sum_{q_1\in A_2(d_1)}C_{(A_1,A_2)}(n_1,q_1)\dots\sum_{q_k\in A_2(d_k)}C_{(A_1,A_2)}(n_k,q_k).
\end{equation*}
Rearranging the terms in this sum we get, \begin{align*}
    &\sum_{q_1\dots q_k=1}^\infty C_{(A_1,A_2)}(n_1,q_1)\dots C_{(A_1,A_2)}(n_k,q_k)\sum_{q_1\in A_2(d_1)\dots q_k\in A_2(d_k)}\frac{\mu_{(A,k)}\times_{A_1}f(d_1,\dots,d_k)}{d_1\dots d_k}\\
        &\quad =\sum_{q_1,\dots,q_k=1}^\infty a(q_1,\dots, q_k)C_{(A_1,A_2)}(n_1,q_1)\dots C_{(A_1,A_2)}(n_k,q_k) 
\end{align*}
It is possible to rearrange the infinite series without altering the sum because of its absolute convergence as shown below,
    \begin{align*}
        &\sum_{q_1,\dots,q_k=1}^\infty|a(q_1,\dots,q_k)||C_{(A_1,A_2)}(n_1,q_1)|\dots |C_{(A_1,A_2)}(n_k,q_k)|\\
        &\leq \sum_{\substack{d_1,\dots,d_k=1\\q_1,\dots,q_k=1\\
        q_1\in A_2(d_1)\dots q_k\in A_2(d_k)}}^\infty \frac{|\mu_{(A_1,k)}\times_{A_1}f(d_1,\dots,d_k)|}{d_1\dots d_k}|C_{(A_1,A_2)}(n_1,q_1)|\dots |C_{(A_1,A_2)}(n_k,q_k)|\\
        &=\sum_{d_1\dots d_k=1}^\infty \frac{|\mu_{(A_1,k)}\times_{A_1}f(d_1,\dots, d_k)|}{d_1\dots d_k}\sum_{\substack{q_1\in A_2(d_1)}}|C_{(A_1,A_2)}(n_1,q_1)|\dots \sum_{\substack{q_k\in A_2(d_k)}}|C_{(A_1,A_2)}(n_k,q_k)|\\
        &\leq n_1\dots n_k\sum_{d_1,\dots, d_k=1}^\infty \frac{|\mu_{(A_1,k)}\times_{A_1} f(d_1,\dots, d_k)|}{d_1\dots d_k}2^{\omega(d_1)+\dots+\omega(d_k)}< \infty
    \end{align*}
\end{proof}

We use the notation $(n_1,\dots,n_k)_{(A_1,k)}$ to denote $\max\{d: d\in A_1(n_1)\cap\dots\cap A_1(n_k)\}$ which is a generalization of $(n_1,\dots, n_k)$ defined in \cite{T_th_2017_multivariable}. 
Next, we consider the case $f(n_1,\dots,n_k)=g((n_1,\dots,n_k)_{(A_1,k)})$.

\begin{thm}\label{expansion2}
    Let $g:\mathbb{N}\rightarrow \mathbb{C}$ be an arithmetic function and $A_1,A_2 \in \mathbb{A}$ such that $A_2\leq A_1$. Assume that,\begin{equation*}
        \sum_{n=1}^\infty 2^{k\omega(n)}\frac{|\mu_{A_1}\times_{A_1}g(n)|}{n^k}<\infty
    \end{equation*}
    Then for every $n_1,\dots, n_k\in \mathbb{N}$,\begin{equation*}
        g((n_1,\dots, n_k)_{A_1,k})=\sum_{q_1,\dots, q_k=1}^\infty a(q_1,\dots, q_k)C_{(A_1,A_2)}(n_1,q_1)\dots C_{(A_1,A_2)}(n_k,q_k)
    \end{equation*}
    is absolutely convergent where,\begin{equation*}
        a(q_1,\dots, q_k)=\sum_{\substack{n=1\\q_1,\dots,q_k\in A_2(n)}}^\infty\frac{\mu_{A_1}\times_{A_1}g(n)}{n^k}
    \end{equation*}
\end{thm}
\begin{proof}
    First note that the following holds for regular convolution,\begin{equation*}
        g((n_1,\dots, n_k)_{(A_1,k)})=\sum_{d\in A_1(n_1)\cap\dots\cap A_1(n_k)}\mu_{A_1}\times_{A_1}g(d)
    \end{equation*}
    It follows from this and $f(n_1,\dots,n_k)=g((n_1,\dots,n_k)_{(A_1,k)})$ 
    that,
    \begin{equation*}
        \mu_{(A_1,k)}\times_{A_1}f(n_1,\dots, n_k)=\begin{cases}
            \mu_{(A_1,k)}\times_{A_1}g(n), &\text{ if }n_1=\dots=n_k=n\\
            0, &\text{ otherwise}
        \end{cases}
    \end{equation*}
    Using this fact and Theorem \ref{expansion1},\begin{equation*}
        a(q_1,\dots, q_k)=\sum_{\substack{n=1\\q_1,\dots, q_k\in A_2(n)}}^\infty\frac{\mu_{A_1}\times_{A_1}g(n)}{n^k}
    \end{equation*}
\end{proof}

\begin{rema}
    While it is true that $q_1,\dots, q_k\in A_2(n) \text{ implies } \lcm(q_1,\dots, q_k)=Q\in A_2(n)$, the converse need not be true. For instance, if $A_2=U$, where $U(n)$ is the set of unitary divisors of $n$, then for prime $p$, $[p^2,p^4]=p^4\in U(p^4)$ but $p^2\notin U(p^4)$.
\end{rema}

Since any arbitrary regular $A$-function $A$ is such that $U\leq A$ (see Figure \ref{fig:struture A func}) so we can state the following corollary.
\begin{cor}\label{cor:expansion}
    Let $g:\mathbb{N}\rightarrow \mathbb{C}$ be an arithmetic function and $A_1$ be a regular $A$-function. Assume that,
    \begin{equation*}
        \sum_{n=1}^\infty 2^{k\omega(n)}\frac{|\mu_{A_1}\times_{A_1}g(n)|}{n^k}<\infty.
    \end{equation*}
     Then for every $n_1,\dots, n_k\in \mathbb{N}$,\begin{equation*}
        g((n_1,\dots, n_k)_{A_1,k})=\sum_{q_1,\dots, q_k=1}^\infty a(q_1,\dots, q_k)C_{(A_1,U)}(n_1,q_1)\dots C_{(A_1,U)}(n_k,q_k)
    \end{equation*}
    is absolutely convergent where,\begin{equation}\label{eq:coefficient Corollary}
        a(q_1,\dots, q_k)=\sum_{\substack{n=1\\q_1,\dots,q_k\in U(n)}}^\infty\frac{\mu_{A_1}\times_{A_1}g(n)}{n^k}.
    \end{equation}
\end{cor}
\begin{proof}
    This is immediate from Theorem \ref{expansion2}.
\end{proof}
Since in Equation \ref{eq:coefficient Corollary}, $q_1,\dots, q_k\in U(n)$ is a condition in determining the coefficients, so it would be helpful to know about the existence of such an $n$ before iterating over several choices. Therefore, we state a simple but important lemma that will be used in determining the coefficients to be considered for the expansion of arithmetic functions.

\begin{lem}\label{lemma:exitence of coefficient}
     Let $A$ be a regular $A$-function.
    Then the following are equivalent:
    \begin{enumerate}
        \item \label{lemma part 1} There exists $n$ such that $q_1,\dots, q_k\in A(n)$.
        \item \label{lemma part 2} $q_1,\dots, q_k\in A(\textnormal{lcm}(q_1,\dots,q_k))$
    \end{enumerate}
\end{lem}
\begin{proof}
    We only prove (\ref{lemma part 1}) implies (\ref{lemma part 2}) as the converse is obvious.
    If there exists $n$ such that $q_1,\dots, q_k\in A(n)$ then clearly $n$ is a multiple of $\lcm(q_1,\dots, q_k)$. Note that if for a prime $p$, if $p^s\parallel \lcm(q_1,\dots, q_k)$ then there exists $j$ such that $p^s\parallel q_j$ and form the multiplicativity (see property (1) in Definition \ref{regular A function}) of regular $A$-functions it follows that $p^s \in A(n)$. This implies $\lcm(q_1,\dots, q_k)\in A(n)$ and from property (3) in Definition \ref{regular A function} $q_1,\dots, q_k\in A(n)$.s
\end{proof}
When $A=U$, Lemma \ref{lemma:exitence of coefficient} gives us a condition under which the coeffcient $a(q_1,\dots,q_k)$ in Corollary \ref{cor:expansion} is non-trivial in the sense that there are terms in the summation \[
\sum_{\substack{n=1\\q_1,\dots,q_k\in U(n)}}^\infty\frac{\mu_{A_1}\times_{A_1}g(n)}{n^k}.
\] 
In the following, we will apply Theorem \ref{expansion2} for certain regular $A$-function analogue of arithmetic function.
For $s \in \mathbb{R}$, define $\sigma_{A}^s(n)$ as an analogue of generalized sum of divisor function for regular $A$-function as:
\begin{equation*}
    \sigma_A^s(n)=\sum_{d\in A(n)}d^s
\end{equation*}
We obtain the following corollary to Theorem \ref{expansion2}.
\begin{cor}\label{sigma_A expansion}
    For every $n_1,\dots, n_k \in \mathbb{N}$ the following series is absolutely convergent:\begin{align*}
        \frac{\sigma_A^s((n_1,\dots,n_k)_{(A,k)})}{(n_1,\dots,n_k)^s_{(A,k)}}&=\zeta(s+k)\sum_{\substack{q_1,\dots,q_k=1\\q_1,\dots,q_k\in U(Q)}}^\infty\frac{\varphi_{s+k}(Q)C_{(A,U)}(n_1,q_1)\dots C_{(A,U)}(n_k,q_k)}{Q^{2(s+k)}},
    \end{align*}
where $s\in\mathbb{R}, s+k>1,\ Q=\textnormal{lcm}(q_1,\dots,q_k).$    
\end{cor}
\begin{proof}
    Let $g(n)=\frac{\sigma_A^s(n)}{n^s} $. Note that $g\times_A\mu_A(n)=\frac{1}{n^s}$, and on using Theorem \ref{expansion2} we get \begin{align*}
        a(q_1,\dots,q_k)&=\frac{1}{Q^{s+k}}\sum_{\substack{n=1\\(Q,n)=1 }}^\infty\frac{1}{n^{s+k}}\\
        &=\frac{1}{Q^{s+k}}\left(\prod_{p\in \mathbb{P}}\left(1-\frac{1}{p^{s+k}}\right)\right)^{-1}\prod_{p\mid Q}\left(1-\frac{1}{p^{s+k}}\right)\\
        &=\zeta(s+k)\frac{\varphi_{s+k}(Q)}{Q^{2(s+k)}}
    \end{align*}
    only for $q_1,\dots, q_k \in  U(Q)$.
\end{proof}
\begin{exmp}
    Taking $s=1$ and $k=2$, for $A\in \mathbb{A}$ we have the following from Corollary \ref{sigma_A expansion},
    \begin{equation}
\frac{\sigma_A(\Tilde{n})}{\Tilde{n}}=\zeta(3)\sum_{\substack{q_1,q_2=1\\q_1,q_2\in U(\textnormal{lcm}(q_1,q_2))}}^\infty\frac{\varphi_3(\textnormal{lcm}(q_1,q_2))}{\textnormal{lcm}(q_1,q_2)^6}C_{(A,U)}(n_1,q_1)C_{(A,U)}(n_2,q_2),
\end{equation}
where $\Tilde{n}=\gcd_{(A,A)}(n_1,n_2)$.
\end{exmp}
For $s\in \mathbb{R}$ define,
$$\varphi^s_A(n)=n^s\prod_{p\mid n}\left(1-\frac{1}{p^{s\tau_A\left(p^{\nu_p(n)}\right)}}\right)$$
as a regular $A$-function analog of Jordan totient function (see \cite[p 48]{apostol1998introduction}). Here, $\tau_A(\cdot)$ is the ``type" function in Definition \ref{regular A function} and $\nu_p(\cdot)$ is the well known $p$-adic valuation function. 
\begin{cor}
    For every $n_1,\dots, n_k \in \mathbb{N}$ the following series is absolutely convergent:
    \begin{equation*}
       \frac{\varphi_A^s((n_1,\dots, n_k)_{(A,k)})}{(n_1,\dots, n_k)_{(A,k)}^s} =\prod_{p\in \mathbb{P}}\left(1-\sum_{i\in S_p}\frac{1}{p^{(s+k)i}}\right)\sum_{\substack{q_1.\dots,q_k=1\\q_1,\dots,q_k\in U(Q)}}^\infty
\frac{\mu_A(Q)C_{(A,U)}(n_1,q_1)\dots C_{(A,U)}(n_k,q_k)}{Q^{(s+k)}\prod_{p\mid Q}(1-\sum_{i\in S_p}(1/p^{(s+k)i}))}    \end{equation*}
where $S_p=\{i: p^i \text{ is A primitive}\}$, $s\geq 1, Q=\lcm(q_1,\dots,q_k).$
\end{cor}
\begin{proof}
    Note that if $g=\varphi_A^s(n)/n^s$ then $g\times_A\mu_A(n)=\frac{\mu_A(n)}{n^s}$. So,\begin{align*}
        \sum_{\substack{m=1\\(m,Q)=1}}^\infty \frac{\mu_A(mQ)}{(mQ)^{s+k}}&=\frac{\mu_A(Q)}{Q^{(s+k)}}\sum_{\substack{m=1\\(m,Q)=1}}^\infty\frac{\mu_A(m)}{m^{(s+k)}}\\
        &=\frac{\mu_A(Q)}{Q^{(s+k)}}\prod_{p\in \mathbb{P}}\left(1-\sum_{i\in S_p}\frac{1}{p^{(s+k)i}}\right)\prod_{p\mid Q}\left(1-\sum_{i\in S_p}\frac{1}{p^{(s+k)i}}\right)^{-1}
    \end{align*}
\end{proof}

%% file: References.tex
\bibliographystyle{acm}
\bibliography{References}